\documentclass{amsart}
\usepackage{verbatim}
\usepackage{rotating} 
\usepackage{amssymb}

\newtheorem{theorem}{Theorem}[section]

\newtheorem{prop}[theorem]{Proposition}

\newtheorem{lem}[theorem]{Lemma}

\newtheorem{cor}[theorem]{Corollary}

\theoremstyle{plain}
\numberwithin{equation}{theorem}

\theoremstyle{remark}

\begin{document}
\bibliographystyle{plain}

\title{Primitivity of prime countable-dimensional regular algebras}
\date{\today}

\author{Pere Ara}
\address{Departament de Matem\`atiques, Universitat Aut\`onoma de Barcelona, 08193 Bellaterra (Barcelona), Spain}
\email{para@mat.uab.cat}

\author{Jason P.~Bell}
\thanks{The first-named author was supported by DGI MICIIN-FEDER
MTM2011-28992-C02-01, and by the Comissionat per Universitats i
Recerca de la Generalitat de Catalunya. 
The second-named author was supported by NSERC grant 31-611456.}

\subjclass[2010]{16E50 (primary), 16D60, 16N60 (secondary)}
\keywords{von Neumann regular rings, primitive rings, prime rings, idempotents, extended centroid, multiplier rings.}%
\address{Department of Pure Mathematics, University of Waterloo, Waterloo, Canada}
\email{jpbell@uwaterloo.ca}

\bibliographystyle{plain}

\begin{abstract} Let $k$ be a field and let $R$ be a countable dimensional prime von Neumann regular $k$-algebra.  We show that $R$ is primitive, answering a special case of a question of Kaplansky.
\end{abstract}
\maketitle

\section{Introduction}
A ring $R$ is \emph{von Neumann regular} if it has the property that for every $a\in R$ there is 
some $x\in R$ such that $a=axa$.  If we have $a=axa$ then $ax$ is necessarily idempotent and 
so one of the interesting consequences of this condition is that principal left and 
right ideals of $R$ are generated by idempotents.  Kaplansky \cite[p. 2]{K} asked whether 
being prime and being primitive are equivalent properties for von Neumann regular rings.

Fisher and Snider \cite{FS} showed that Kaplansky's question has an affirmative answer for a von Neumann regular 
ring $R$ if $R$ is either a countable ring or if $R$ has a countable downward cofinal set $\mathcal S$ 
of nonzero two-sided ideals (meaning that every nonzero two-sided ideal contains an element of $\mathcal S$). 
Domanov \cite{D} later constructed 
a counter-example by giving an example of a von Neumann regular prime-but-not-primitive group algebra.  
Further examples were later given by using Leavitt path algebra constructions \cite{ABR}.  

These counter-examples all have the property that they are uncountably infinite dimensional over their base fields and 
it is therefore natural to ask whether Kaplansky's question has an affirmative answer if one adds the additional 
hypothesis that the algebra be countable dimensional.  When the base field is countable this follows from 
Fisher and Snider's theorem since such algebras are necessarily countable as rings.  On the other hand, 
Fisher and Snider's work says nothing about von Neumann regular algebras over uncountable fields 
unless one has the countable cofinality condition on the set of two-sided ideals.

In this paper, we answer Kaplansky's question for countable dimensional rings.
\begin{theorem} Let $k$ be a field and let $R$ be a prime von Neumann regular $k$-algebra with ${\rm dim}_k(R)\le \aleph_0$.  Then $R$ is both left and right primitive.  
\label{thm: main}
\end{theorem}

 Indeed, the class of algebras to which our argument applies is a much larger class than the class
of prime regular rings, namely the class of prime algebras such that every nonzero right ideal $I$ of $R$ contains a 
nonzero idempotent $e$ with the property that the center $Z(eRe)$ is a field.  Thus we shall show that such algebras 
are automatically primitive when $R$ is countable dimensional (Theorem \ref{thm:the-one}). The class    
of algebras $R$ such that every nonzero right ideal of $R$ contains a nonzero idempotent is very large. It contains,
amongst others, 
all exchange semiprimitive algebras \cite{Nic}, all Leavitt path algebras of arbitrary graphs with the property (L), see e.g. 
\cite[Theorem 1]{Ranga},
all Kumjian-Pask algebras of aperiodic row-finite higher-rank graphs without sources, see \cite[Proposition 4.9]{ACHR}, 
and all abelianized Leavitt path algebras $L^{{\rm ab}}_K(E,C)$ of finite bipartite separated graphs $(E,C)$ 
with condition (L) \cite[Theorem 10.10]{AE}.
The condition about the centers $Z(eRe)$ of the corner algebras holds in many examples, and is automatic when the prime algebra $R$ 
is centrally closed, which is the case for prime Leavitt path algebras of finite graphs, by \cite[Theorem 3.7]{CMMSS}.

The outline of the paper is as follows.  In Section 2, we give some background on the extended centroid and on multiplier rings.  We use 
this material to show that if $R$ is a countable dimensional prime von Neumann regular algebra (or satisfies the weaker hypothesis
alluded to in the previous paragraph), whose base field is uncountable, 
then the extended centroid of $R$ is an algebraic extension of the base field. 
This fact is then used in Section 3 to prove our main result.

\section{The extended centroid and multiplier rings}
Let $R$ be a semiprime ring (we do not assume that $R$ has identity).  The \emph{multiplier ring} $M(R)$ consists of all 
tuples $(f,g)$ where $f:R\to R$ is a left $R$-module homomorphism, $g:R\to R$ is a right $R$-module homomorphism and they satisfy the balanced condition
$f(x)y= xg(y)$ for all $x,y\in R$.  Addition in this ring is just given by ordinary coordinate-wise addition of maps and multiplication is given by the rule
$$(f_1,g_1)\cdot (f_2,g_2):=(f_2\circ f_1, g_1\circ g_2).$$
It is straightforward to check that with these operations $M(R)$ is a ring with identity given by $({\rm id},{\rm id})$.  Moreover, $R$ embeds 
as a subring in $M(R)$ via the rule $r\mapsto (\mathcal{R}_r,\mathcal{L}_r)$, where $\mathcal{L}_r$ and $\mathcal{R}_r$ are respectively left and right multiplication by $r$.

We need the following well-known lemma. We include a proof for the reader's convenience.

\begin{lem}
 \label{lem:Morita-inv} Let $e$ be an idempotent of a prime ring $R$. Then $Z(M(ReR))\cong Z(eRe)$.  
\end{lem}

\begin{proof}
 Define $\phi \colon Z(M(ReR))\to Z(eRe)$ by $\phi (x) =exe$, where we identify 
 $ReR$ with a subring of $M(ReR)$ in the natural way. Observe that $xe\in ReR$ because $x\in M(ReR)$ and $e\in ReR$, and thus
$exe = e(xe) e\in  eRe$.  Since $x\in Z(M(ReR))$ we have $xe=ex=exe\in eRe$. Using this, we see that $\phi$ is a unital ring homomorphism.
If $\phi (x) =0$, then $xe= 0$ and so $x\Big( \sum_i r_ies_i\Big) = \sum _i r_i (xe)s_i = 0$ for all $r_i,s_i\in R$, which in turn implies that $x= 0$.
This shows that $\phi$ is injective. To show surjectivity, take $x\in Z(eRe)$ and define $f_x\colon ReR \to ReR$ by
$$f_x \Big( \sum_i r_ies_i\Big) = \sum _i r_i xs_i.$$ To show that $f_x$ is well-defined, suppose that $\sum _i r_ies_i =0$.
Then for all $t\in R$ we have, using that $x\in Z(eRe)$,
$$\Big( \sum_i r_i xs_i\Big) te= \sum _i r_ix(es_ite) = \Big( \sum _i r_ies_i\Big) (tex) = 0.$$  
This shows that $\sum _i r_ixs_i =0$. Thus $f_x$ is a well-defined $R$-bimodule homomorphism and thus $(f_x,f_x) \in Z(M(I))$.
Clearly $\phi ((f_x,f_x)) =x$. This concludes the proof.  
\end{proof}

The \emph{symmetric ring of quotients} of a semiprime ring $R$ can be defined as follows.  We let $\mathcal{T}$ denote the 
collection of triples $(f,g,I)$, where $I$ is an essential ideal of $R$ and $f,g:I\to R$ are maps with $f$ a left $R$-module homomorphism, $g$ a right $R$-module homomorphism 
and such that $f(x)y=xg(y)$ for all $x,y\in I$.  We put an equivalence relation $\sim$ on $\mathcal{T}$ by declaring 
that $(f_1,g_1,I_1)\sim (f_2,g_2,I_2)$ if $f_1$ and $f_2$ agree on $I_1\cap I_2$ and $g_1$ and $g_2$ agree on $I_1\cap I_2$.  It is straightforward 
to check that this gives an equivalence relation on $\mathcal{T}$ and we let $[(f,g,I)]$ denote the equivalence class 
of $(f,g,I)\in \mathcal{T}$.  We then define the symmetric ring of quotients, $Q_s(R)$, of $R$ to be the collection 
of equivalences classes $[(f,g,I)]$ with addition and multiplication defined by
\begin{equation}
[(f_1,g_1,I_1)]+[(f_2,g_2,I_2)]=[(h_1,h_2,I_1\cap I_2)]\end{equation}
and
\begin{equation}
[(f_1,g_1,I_1)]\cdot [(f_2,g_2,I_2)]=[(k_1,k_2,I_2 I_1)]
\end{equation} where
$h_1$ is the restriction of $f_1+f_2$ to $I_1\cap I_2$; $h_2$ is the restriction of $g_1+g_2$ to $I_1\cap I_2$; $k_1=f_2\circ f_1$ and 
$k_2=g_1\circ g_2$ on $I_2I_1$.   Then $Q_s(R)$ is a ring with identity given by the class $[({\rm id},{\rm id}, R)]$.  Furthermore, $R$ 
embeds in $Q_s(R)$ as a subring via the rule $r\mapsto [(\mathcal{R}_r,\mathcal{L}_r,R)]$, and we identify $R$ with its image in $Q_s(R)$.  
We define the \emph{extended centroid} of $R$, $C(R)$, to be the centre of $Q_s(R)$.   We note that if $R$ is a prime ring, then the extended centroid 
of $R$ is a field \cite[Corollary 2.1.9]{AM}.  

In the case of a prime ring $R$, the extended centroid has an alternative description as follows (see e.g. \cite[page 359]{A}).  
We let $\mathcal{B}$ denote the collection of pairs $(f,I)$, where $I$ is a nonzero ideal of $R$ and $f:I\to R$ is an $(R,R)$-bimodule 
homomorphism.  We place an equivalence relation on $\mathcal{B}$ by declaring that $(f,I)\sim (g,J)$ if $f=g$ on $I\cap J$ and we let $[(f,I)]$ 
denote the equivalence class of $(f,I)\in \mathcal{B}$.  Then the set of equivalence classes forms a ring with addition given 
by $[(f,I)]+[(g,J)]=[(f+g,I\cap J)]$ and $[(f,I)]\cdot [(g,J)] = [(f\circ g, JI)]$.  

We note that if $R$ is prime and $a,b\in R$ 
are nonzero with $axb=bxa$ for every $x\in R$ then we have a bimodule isomorphism $f:(a)\to (b)$ given by $\sum x_i a y_i \mapsto \sum x_i b y_i$.  
To check this it is sufficient to check that 
$\sum x_i a y_i = 0$ if and only if $\sum x_i b y_i=0$.  We check one direction; 
the other direction follows from symmetry.  Suppose that $\sum x_i a y_i =0$.  Then $$\left( \sum x_i a y_i\right)rb=0$$ for every $r\in R$.  Using the fact that $axb=bxa$ for every $x\in R$ gives $$\left( \sum x_i b y_i\right)ra=0$$ for every $r\in R$.  Since $a$ is nonzero and $R$ is prime, we see that 
$\sum x_i b y_i=0$.  
%
%
%
%
%
%

\begin{lem}
\label{lem:enoughidemps}
Let $k$ be a field and let $R$ be a prime $k$-algebra such that every nonzero ideal of $R$ contains a nonzero idempotent $e$
such that $Z(eRe)$ is a field.
  If ${\rm dim}_k(R)<|k|$ then the extended centroid of $R$ is an algebraic extension of $k$.
\label{lem: 0}
\end{lem}
\begin{proof} We have $C(R)=Z(Q_s(R))$. Let $\mathcal F $ be the family of nonzero ideals $I$ of $R$ of the form $I=ReR$, where $e$ is an idempotent in $R$ such that 
$Z(eRe)$ is a field.
If $J$ is a nonzero ideal of $R$, then by hypothesis there is $I\in \mathcal F$ such that $I\subseteq J$. Note that $I= I^2$ for all $I\in \mathcal F$.
Observe also that, if $I\subseteq J$ are idempotent nonzero ideals of a prime ring $T$, then there is a natural embedding $M(J)\to M(I)$, and this carries central elements 
to central elements. Hence, given a downward directed system of idempotent nonzero ideals, one gets a directed system of embeddings of their multiplier rings, and a directed system of 
embeddings of the centers of these.

By the above remarks and the proof of Proposition 2.1.3 of \cite{AM}, we may write $C(R)$ as 
$$Z\left( \varinjlim_{ I\in \mathcal F} M(I)\right),$$
where $\mathcal F$ is partially ordered by reverse inclusion. Since $Z\left( \varinjlim_{ I\in \mathcal F} M(I)\right)=  \varinjlim_{ I\in \mathcal F} Z(M(I))$, it is sufficient to show 
that $Z(M(I))$ is a field which is an algebraic extension of $k$ whenever $I\in \mathcal F $, as this will give 
that $C(R)$ is  a directed limit of algebraic field extensions of $k$.  Let $I\in \mathcal F $. Then, by definition, there is an idempotent $e$ in $I$ such that 
$I=ReR$ and $Z(eRe)$ is a field.  By Lemma \ref{lem:Morita-inv},
we have that $Z(M(I))$ is isomorphic as $k$-algebra to $Z(eRe)$, and thus $Z(M(I))$ is a field. 
Since ${\rm dim}_k(eRe)<|k|$, we have that $Z(M(I))$ is an algebraic extension 
of $k$. Indeed, if $t$ is a transcendental element of $Z(M(I))$, then $\{ (t-\alpha)^{-1} : \alpha \in k \}$
is a set of linearly independent elements of $Z(M(I))$, of cardinality $|k|$, which contradicts the fact that
${\rm dim}_k \, Z(M(I))=  {\rm dim}_k \,  Z(eRe)<|k|$. 
As $C(R)$ is a directed limit of algebraic extensions of $k$, we conclude that $C(R)$ is algebraic over $k$.
\end{proof}

\begin{cor}
\label{cor:extcenalgforprimereg}
 Let $k$ be a field and let $R$ be a prime regular $k$-algebra. If ${\rm dim}_k(R)<|k|$ then the extended centroid of $R$ is an algebraic extension of $k$.
 \end{cor}

 \begin{proof}
  If $e$ is a nonzero idempotent of $R$ then $eRe$ is a nonzero prime regular ring, and so $Z(eRe)$ is a field.
  Therefore, the result follows from Lemma \ref{lem:enoughidemps}.
 \end{proof}

 \section{Countable dimensional prime regular algebras}

\begin{lem} Let $k$ be a field, let $R$ be a prime $k$-algebra, and let $V$ be a nonzero finite dimensional subspace of $R$.  Suppose that the extended centroid of $R$ is an algebraic extension of $k$.  Then there exists a nonzero element $r\in R$ such that $r\in I$ for all two-sided ideals $I$ such that $V\cap I\neq (0)$.
\label{lem: 1}
\end{lem}
\begin{proof}
We prove this by induction on the dimension of $V$.  If ${\rm dim}(V)=1$, there is nothing to prove.  We thus assume that the claim holds whenever ${\rm dim}(V)<n$ and we consider the case that ${\rm dim}(V)=n$ with $n\ge 2$.  Let $\{r_1,\ldots ,r_n\}$ be a basis for $V$ and let $\mathcal{S}$ denote the collection of two-sided ideals in $R$ with the property that $I\cap V\neq (0)$.  Then there exist maps
 $c_1,\ldots ,c_n: \mathcal{S}\to k$ such that for each $I\in \mathcal{S}$ we have
 $$\sum_{i=1}^n c_i(I) r_i \in I$$ and such that for each $I\in \mathcal{S}$ there exists some $i=i(I)\in \{1,\ldots ,n\}$ such that $c_i(I)\neq 0$.   We now consider two cases.
 \vskip 2mm
 {\em Case I:} there is some $j$ and some $x\in R$ such that $r_jxr_1\neq r_1xr_j$.
 \vskip 2mm
 We let $T$ denote the set of elements $(b_2,\ldots ,b_n)\in k^{n-1}$ such that 
 $$\sum_{i=2}^n b_i (r_i x r_1 - r_1 x r_i) = 0.$$  Then $T$ is a proper subspace of $k^{n-1}$, since $r_j x r_1-r_1xr_j\neq 0$.  
We let $$\mathcal{S}_0 = \{ I\in \mathcal{S}\colon (c_2(I),\ldots ,c_n(I))\in T\}.$$  For $I\in \mathcal{S}$ we have
 $$\left( \sum_i c_i(I) r_i \right) x r_1 - r_1 x\left( \sum_i c_i(I) r_i \right) \in I.$$  We may rewrite this element as
 $$\sum_{i=2}^n c_i(I) (r_i x r_1 - r_1 x r_i).$$  Furthermore, for $I\in \mathcal{S}\setminus \mathcal{S}_0$, 
 we have $I\cap W\neq (0)$, where $W$ is the span of $\{r_i x r_1 - r_1 x r_i\colon i=2,\ldots ,n\}$.  Let
 $$W':=\left\{ \sum_{i=1}^n b_i r_i \colon (b_1,b_2,\ldots ,b_n)\in k^n~{\rm and}~(b_2,\ldots ,b_n)\in T\right\}.$$ 
 It follows that if $I\in \mathcal{S}$ then either $I\cap W\neq (0)$ or $I\cap W'\neq (0)$.  Notice that $W'$ has dimension $1+{\rm dim}(T)<n$, since $T$ is a proper subspace of $k^{n-1}$.  
 
 Since both $W$ and $W'$ are nonzero and have dimension less than $n$, the inductive hypothesis gives that there exist nonzero elements 
 $a,b\in R$ such that  either $a$ or $b$ is in $I$ whenever $I\in \mathcal{S}$.   Since $R$ is prime, there exists some $y\in R$ such that 
 $ayb\neq 0$.  We thus see that $ayb\in I$ whenever $I\cap V\neq (0)$ giving us the claim in this case.
 \vskip 2mm
 {\em Case II:} $r_1xr_j=r_j x r_1$ for all $x\in R$ and all $j=2,3,\ldots , n$.
 \vskip 2mm
We use the description of $C(R)$ as equivalence classes of 
 bimodule homomorphisms.  We let $F$ denote the extension of $k$ generated by the elements $[(f_j, (r_j))]$ for $j=1,\ldots ,n$, 
 %
%
%
%
 %
 %
 where $f_j:(r_1)\to R  $ is given by $x r_1 y \mapsto xr_jy$ and extending via linearity.  We must check that each $f_j$ is well-defined.   To see this, it is sufficient to check that whenever the expression 
 $\sum_i x_i r_1 y_i$ is equal to zero, we necessarily have 
 $\sum_i x_i r_j y_i=0$.  Thus, let us suppose towards a contradiction, that we can find elements $x_1,\ldots ,x_d,y_1,\ldots ,y_d\in R$ such that
 $$\sum_{i=1}^d x_i r_1 y_i =0\qquad {\rm and}\qquad \sum_{i=1}^d x_i r_j y_i \neq 0.$$
 Since $R$ is prime, there exists $y\in R$ such that $$\left(\sum_{i=1}^d x_i r_j y_i \right)y r_1\neq 0.$$
 But we have $r_j y_i y r_1 = r_1 y_i y r_j$ for all $i$ and $j$, and so
 $$0=\left(\sum_{i=1}^d x_i r_1 y_i \right)y r_j= \left(\sum_{i=1}^d x_i r_j y_i \right)y r_1\neq 0,$$ which is a contradiction.  Thus we see that the maps $f_j$ are well-defined.  
 
 Since the extended centroid of $R$ is algebraic, and $F$ is finitely generated as a field extension of $k$, $F$ is a finite extension  of $k$.  We fix a basis $\{h_1,\ldots ,h_m\}$ for $F$ as a $k$-vector space.  For each $i\in \{1,\ldots ,m\}$ there is a nonzero element $b_i$ of $R$ with the property that $h_i$ sends $b_i$ into $R$.  We pick a nonzero $b\in R$ that is in the intersection of the ideals generated by $b_1,\ldots ,b_m$.  Then by construction, $bF\subseteq R$.  Notice that for $I\in \mathcal{S}$, we have
 $$\sum_i c_i(I)r_i\in I.$$  We may rewrite this as
 $$\left(\sum_i c_i(I)f_j\right)(r_1)\in I.$$  Equivalently, there is a map $\beta: \mathcal{S}\to F\setminus \{0\}$ such that $\beta(I)r_1 \in V\cap I$.  Since $\beta(I)^{-1}\in F$, we have that $b\cdot \beta(I)^{-1}\in R$.  Furthermore, $R$ is prime and hence there is some nonzero element $t$ such that $btr_1\neq 0$.  Then $(b\cdot \beta(I)^{-1} t)(\beta(I) r_1) = btr_1\in I$ whenever $I\cap V\neq (0)$ giving us the claim in the remaining case.
\end{proof}
We note that the conclusion to the statement of Lemma \ref{lem: 1} need not hold if the extended centroid of $R$ is not an algebraic extension of the base field.  As a simple example, let $k$ be an infinite field and take $R=k[x]$ and let $V=kx+k$.  Then for each $\alpha\in k$, we have $R(x-\alpha)\cap V\neq (0)$, but the intersection of $R(x-\alpha)$ as $\alpha$ ranges over $k$ is zero.

We now prove our main theorem.  In the proof, we make use of a result of Fisher and Snider that guarantees primitivity of a prime ring under 
the conditions that each nonzero {\it right} ideal of the ring has a nonzero idempotent and such that there is a countable downward cofinal set of nonzero ideals.  
We recall that a set $\mathcal{S}$ of nonzero ideals in a ring $R$ is \emph{downward cofinal} if each nonzero ideal of $R$ contains an ideal in $\mathcal{S}$.
\begin{theorem} 
\label{thm:the-one}
Let $k$ be a field and let $R$ be a prime $k$-algebra  such that each nonzero right ideal of $R$
contains a nonzero idempotent $e$ such that $Z(eRe)$ is a field. Suppose that ${\rm dim}_k(R)\le \aleph_0$.  
Then $R$ is both left and right primitive.  
\end{theorem}
\begin{proof} If $k$ is a countable field, then $R$ is a countable ring and hence $R$ is right and left primitive by  \cite[Theorem 1.1]{FS}.  
Thus it is no loss of generality to assume that $k$ is uncountable.  Thus in this case we have that the extended centroid 
of $R$ is an algebraic extension of $k$ 
by Lemma \ref{lem:enoughidemps}.

Since $R$ is countably-infinite or finite dimensional, there exists an increasing sequence of finite-dimensional nonzero subspaces $$V_1\subseteq V_2 \subseteq \cdots $$ such that
$$R=\bigcup_i V_i.$$  
By Lemma \ref{lem: 1}, there exist nonzero elements $a_i\in R$ such that $a_i\in I$ whenever $I\cap V_i\neq (0)$.  
We let $J_i=(a_i)$ denote the two-sided ideal generated by $a_i$.  Then since every nonzero ideal in $R$ must intersect 
some $V_i$ we see that the collection $\{J_i\colon i\ge 1\}$ is a countable downward cofinal set of ideals in $R$.  Thus by 
Fisher and Snider \cite[Theorem 1.1]{FS}, we see that $R$ is both left and right primitive.
\end{proof}

\medskip

We can now obtain a proof of Theorem \ref{thm: main}.

\medskip

{\it Proof of Theorem \ref{thm: main}:}  Recall that $Z(eRe)$ is a field for every nonzero idempotent
$e$ of a prime von Neumann regular ring $R$. Hence the proof follows immediately from Theorem \ref{thm:the-one}.
\qed

\medskip

Since finitely generated algebras are necessarily at most countably infinite dimensional over their base fields, 
we obtain the following result as an immediate corollary of Theorem \ref{thm:the-one}.
\begin{cor} Let $k$ be a field and let $R$ be a finitely generated prime $k$-algebra such that each nonzero right ideal of $R$
contains a nonzero idempotent $e$ such that $Z(eRe)$ is a field.  Then $R$ is both left and right primitive.  
\end{cor}

We close by recording the equivalence between various of the hypothesis that have played a role in the proof of our results.

\begin{prop}
 \label{prop:recordin-equivs}
 Let $k$ be an uncountable field and let $R$ be a prime countable dimensional $k$-algebra such that each nonzero 
 (two-sided) ideal
 of $R$ contains a nonzero idempotent. Then the following conditions are equivalent:
 \begin{enumerate}
  \item The extended centroid $C(R)$ is an algebraic field extension of $k$.
  \item ${\rm dim}_k\, C(R) \le \aleph_0$.
  \item For each nonzero idempotent $e$ of $R$ we have that $Z(eRe)$ is a field.
  \item Each nonzero ideal of $R$ contains a nonzero idempotent $e$ such that $Z(eRe)$ is a field.
   \end{enumerate}
\end{prop}

\begin{proof}
 (i)$\implies $(ii). As in the proof of Theorem \ref{thm:the-one}, $R$ has a countable downward cofinal
 chain of two-sided ideals $J_i$. Since each contains a nonzero idempotent, we can obtain from them 
 a countable chain of idempotent-generated ideals $Re_iR$ (with $e_i^2=e_i$), cofinal among all nonzero ideals of $R$.
 We have that 
 $$C(R) = \varinjlim_i \,  ZM(Re_iR) .$$
  Since $ZM(Re_iR)\cong Z(e_iRe_i)$ by Lemma \ref{lem:Morita-inv}, we get that $\text{dim}_k (ZM(Re_iR))\le \aleph_0$,
 and consequently $\text{dim}_k \,  C(R) \le \aleph_0$. 
 
  (ii)$\implies $ (i). This follows from the facts that $C(R)$ is a field extension of $k$ and that $k$ is an uncountable field
  (cf. the last part of the proof of Lemma \ref{lem:enoughidemps}). 
 
 (i)$\implies $(iii). Let $e$ be a nonzero idempotent of $R$. The extension $k\subseteq Z(eRe)$ is algebraic, because
 $Z(eRe)\subseteq C(R)$. Since $Z(eRe)$ is a domain, it follows that it must be a field. 
 
 (iii)$\implies $(iv). Obvious.
 
 (iv)$\implies $(i). This follows from Lemma \ref{lem:enoughidemps}.
 \end{proof}

\section*{Acknowledgments}
We thank Efim Zelmanov, who suggested this question to us during a ring theory meeting in Poland.  We also thank the anonymous referee for carefully reading the paper and for providing many helpful suggestions.


\begin{thebibliography}{99}
\bibitem{ABR} G. Abrams, J. Bell and K. M. Rangaswamy, On prime non-primitive von Neumann regular algebras.  To appear in \emph{Trans. Amer. Math. Soc.}
\bibitem{ACHR} G. Aranda Pino, J. Clark, A. an Huef and I. Raeburn,  
Kumjian-Pask algebras of higher-rank graphs, \emph{Trans. Amer. Math. Soc.} {\bf 365} (2013), 3613--3641.
\bibitem{A} P. Ara, The extended centroid of C*-algebras, 
\emph{Arch. Math.} {\bf 54} (1990), 358--364.
\bibitem{AE} P. Ara and R. Exel, Dynamical systems associated to separated graphs, graph algebras, 
and paradoxical decompositions, arXiv:1210.6931v1 [math.OA]. 
\bibitem{AM} P. Ara and M. Mathieu, 
\emph{Local multipliers of $C^*$-algebras.}
Springer Monographs in Mathematics. Springer-Verlag London, Ltd., London, 2003.
\bibitem{CMMSS} M. G. Corrales Garcia, D. Martin Barquero, C. Martin Gonzalez, M. Siles Molina and J. F. Solanilla Hernandez,
Extreme cycles. The center of a Leavitt path algebra,  arXiv:1307.5252v1 [math.RA].
\bibitem{D} O. I. Domanov, A prime but not primitive regular ring, \emph{Uspehi Mat. Nauk} {\bf 32} (6) (1977), 219--220.
\bibitem{FS} J. W. Fisher and R. L. Snider, 
Prime von Neumann regular rings and primitive group algebras, 
\emph{Proc. Amer. Math. Soc.} {\bf 44} (1974), 244--250. 
\bibitem{K} I. Kaplansky, \emph{Algebraic and analytic aspects of operator algebras.} Conference Board of the Mathematical Sciences, Regional Conference Series in Mathematics, Number 1. American Mathematical Society, Providence, RI, 1970.
\bibitem{Nic} W. K. Nicholson, Lifting idempotents and exchange rings, \emph{Trans. Amer. Math. Soc.} {\bf 229} (1977), 269--278.
\bibitem{Ranga} K. M.  Rangaswamy, Leavitt path algebras which are Zorn rings, arXiv:1302.4681v1 [math.RA].


\end{thebibliography}
\end{document}